\documentclass[11pt]{amsart}

\usepackage[a4paper,margin=27mm]{geometry}
\usepackage{amsmath,amssymb,mathtools}
\usepackage{graphicx}
\usepackage{microtype}
\usepackage{tikz}
\usetikzlibrary{arrows.meta,calc,positioning}
\usepackage[section]{placeins}

\definecolor{OmegaBlue}{RGB}{42,101,155}
\definecolor{OmegaFill}{RGB}{214,231,244}
\definecolor{CylinderGray}{RGB}{125,132,139}

\tikzset{
  axis/.style={-{Latex[length=2.2mm]},line width=0.55pt},
  boundary/.style={draw=OmegaBlue!88!black,line width=1.15pt},
  section/.style={draw=OmegaBlue!88!black,line width=0.85pt},
  hidden section/.style={draw=OmegaBlue!78!black,dashed,line width=0.75pt},
  artificial/.style={draw=CylinderGray,dashed,line width=0.65pt},
  note arrow/.style={-{Latex[length=1.8mm]},line width=0.55pt},
  panel title/.style={font=\small\bfseries},
  every node/.style={font=\small}
}

\usepackage[hidelinks]{hyperref}
\usepackage[nameinlink,noabbrev]{cleveref}

\hypersetup{
 pdftitle={Nonconvex Sublevel Sets for the Three-Dimensional Special Lagrangian Equation},
 pdfauthor={Guohuan Qiu}
}

\allowdisplaybreaks
\numberwithin{equation}{section}

\newtheorem{theorem}{Theorem}[section]
\newtheorem{proposition}[theorem]{Proposition}
\newtheorem{lemma}[theorem]{Lemma}
\theoremstyle{remark}
\newtheorem{remark}[theorem]{Remark}

\newcommand{\R}{\mathbb R}
\newcommand{\F}{\mathcal F}
\newcommand{\e}{\varepsilon}
\newcommand{\tr}{\operatorname{tr}}
\newcommand{\ii}{\mathrm i}

\title[Nonconvex sublevel sets for the special Lagrangian equation]
{Nonconvex Sublevel Sets for the Three-Dimensional
Special Lagrangian Equation}

\author{Guohuan Qiu}

\address{Institute of Mathematics, Academy of Mathematics and Systems Science,
Chinese Academy of Sciences, No. 55 Zhongguancun East Road, Beijing 100190,
China}
\email{qiugh@amss.ac.cn}


\begin{document}

\begin{abstract}
For each phase $\Theta\in(\pi/2,\pi)$, we construct a smooth, bounded,
uniformly convex domain in $\mathbb R^3$ for which the zero-Dirichlet
solution of the special Lagrangian equation has a nonconvex negative
sublevel set. 
\end{abstract}

\maketitle

\section{Introduction}

For $A\in\operatorname{Sym}(n)$, define
\[
 \F_n(A)=\sum_{j=1}^n\arctan\lambda_j(A),
\]
where each inverse tangent takes values in $(-\pi/2,\pi/2)$. The equation
\begin{equation}\label{eq:SLE}
 \F_n(D^2u)=\Theta
\end{equation}
is the special Lagrangian equation, which serves as the potential equation for a special Lagrangian gradient graph \cite{HarveyLawson82}. The phase is called critical when
\(
|\Theta|=\frac{(n-2)\pi}{2},
\)
and supercritical when
\(
|\Theta|>\frac{(n-2)\pi}{2}.
\)

Let $\Omega\subset\R^n$ be smooth, bounded, and strictly convex. 
The Dirichlet problem for the special Lagrangian equation was first
treated by Caffarelli, Nirenberg and Spruck for certain elliptic phases
under suitable convexity assumptions on the domain
\cite{CNS}. Harvey and Lawson later established
existence and uniqueness of continuous viscosity solutions for all
branches on appropriately pseudoconvex domains
\cite{HarveyLawsonCPAM}, while subsequent works developed the
classical solvability theory in the critical and supercritical phase
regimes \cite{CollinsPicardWu,luDCDS,BhattacharyaMooneyShankarAJM}.
In particular, if
\[
 \frac{(n-2)\pi}{2}\le |\Theta|<\frac{n\pi}{2},
\]
then
\begin{equation}\label{eq:Dirichlet-intro}
 \begin{cases}
  \F_n(D^2u)=\Theta&\text{in }\Omega,\\
  u=0&\text{on }\partial\Omega
 \end{cases}
\end{equation}
has a unique solution which is smooth in $\Omega$ and Lipschitz continuous
on $\overline\Omega$.

For positive phase, Yuan \cite{Yuan06} proved that the level hypersurface
\[
 \left\{\lambda\in\mathbb{R}^n:
 \sum_{j=1}^n\arctan\lambda_j=\Theta\right\}
\]
in eigenvalue space is convex if and only if
\(
 \Theta\geq \frac{(n-2)\pi}{2}.
\)
This convexity plays an important role in the regularity theory for fully
nonlinear elliptic equations.  It provides the structural convexity
underlying the Evans--Krylov--Safonov theory and its extensions to equations
whose level hypersurfaces are convex; see \cite{caffarelli2000priori}.  The
critical threshold is also sharp from the viewpoint of regularity: smoothness
is known in the critical and supercritical regimes
\cite{warren2009hessian,WY10,WY11}, whereas singular solutions exist in the
subcritical regime
\cite{nadirashvili2010singular,wang2013singular,mooney2024non}.

For positive phase, the maximum principle implies that
\(u<0\) in \(\Omega\).  Since singular solutions are known to exist in
the subcritical regime, we focus on the supercritical phase, where the
Dirichlet problem admits smooth solutions under suitable assumptions.
This leads to the following natural question: does a smooth solution
inherit the convexity of the zero boundary level set?  

The borderline critical case is of particular interest. In dimension
three, when \(\Theta=\pi/2\), the special Lagrangian equation reduces, on
its elliptic branch, to
\[
 \sigma_2(D^2u)=1.
\]
The convexity of the level sets of the zero-Dirichlet solution was proved
by Ma and Xu \cite{MaXu} and later by Salani \cite{Salani}. This result was recently extended to higher
dimensional \(2\)-Hessian equations \cite{li2026brunn}.  In dimension two,
Zhang and Zhou \cite{zhang2023power} established a power-convexity result for
solutions of the special Lagrangian equation with zero Dirichlet data.  On
the other hand, in any dimension, the larger-phase regime
\(
 \Theta\geq\frac{(n-1)\pi}{2}
\)
is much more rigid: on the positive-phase branch it directly implies
\(
 D^2u>0.
\)
These observations indicate that, in dimension three, the interval
\[
 \frac{\pi}{2}<\Theta<\pi
\]
is the only strictly supercritical range in which nonconvex level sets
may occur.  Our result fills
this remaining gap.

The convexity of spatial level sets has been studied extensively for
various fully nonlinear elliptic equations; see
\cite{Kawohl1985,Korevaar1990}.  However, convexity of the domain alone
is not sufficient in general
\cite{HamelNadirashviliSire,MonneauShahgholian}.  In a related geometric
setting, Wang \cite{Wang14} constructed a smooth uniformly convex domain
for which the zero-Dirichlet solution of the constant mean curvature
equation has a nonconvex level set. Here, special Lagrangian equation is actually the minimal surface equation in higher codimension.

\begin{theorem}
\label{thm:three-dimensional}
For every
\[
 \Theta\in\left(\frac{\pi}{2},\pi\right),
\]
there exist a smooth, bounded, uniformly convex domain
\(\Omega_\Theta\Subset\mathbb R^3\) and a unique solution
\[
 u_\Theta\in C^\infty(\Omega_\Theta)
 \cap C^{0,1}(\overline{\Omega_\Theta})
\]
of
\begin{equation}\label{eq:3d-main-problem}
 \begin{cases}
  \F_3(D^2u_\Theta)=\Theta&\text{in }\Omega_\Theta,\\
  u_\Theta=0&\text{on }\partial\Omega_\Theta,
 \end{cases}
\end{equation}
such that there exists a negative regular value \(c_\Theta\) for which
the sublevel set
\[
 \{u_\Theta<c_\Theta\}
\]
is nonconvex.
\end{theorem}

The proof consists of a local construction and a perturbative argument.  The main difficulty is to find a suitable local model.  The leading-order structure of our ansatz is partly inspired by Wang and Yuan \cite{wang2013singular}, while the perturbative argument is similar to those used by Wang \cite{Wang14} and, more recently, by Zhang \cite{zhang2026strictly}.  We begin with a family of two-dimensional Monge--Ampere solutions written in partial Legendre coordinates.  By adding suitable harmonic perturbations, we make the tangential Hessian positive along the central circle in the zero level set, while retaining a negative tangential direction on a lower level set.  A quartic term then closes the local geometry into a bounded uniformly convex domain.  After an anisotropic dilation, the error in the three-dimensional phase is reduced to \(O(\varepsilon^2)\).  Finally, a comparison argument transfers the strict midpoint defect of the local model to the exact Dirichlet solution, thereby producing a nonconvex sublevel set.

\section{Construction of the local model}

Fix
\[
 \Theta\in\left(\frac{\pi}{2},\pi\right)
\]
and set
\begin{equation}\label{eq:3d-parameters}
 \alpha=-\cot\Theta>0,\qquad
 \kappa=\csc\Theta=\sqrt{1+\alpha^2},\qquad
 \gamma=\frac{\alpha}{\kappa}\in(0,1).
\end{equation}
We shall repeatedly use
\begin{equation}\label{eq:3d-parameter-identity}
 (\kappa-\alpha)(\kappa+\alpha)=1,
 \qquad \kappa-\alpha>0.
\end{equation}

\subsection{Exact two-dimensional sections}

\begin{lemma}\label{lem:3d-phase-reduction}
Let $\phi$ be strictly convex in two variables and satisfy
\[
 \det D^2\phi=1.
\]
Then
\[
 B=\alpha I+\kappa D^2\phi
\]
is positive definite and
\[
 \F_2(B)=\Theta.
\]
\end{lemma}

\begin{proof}
Since $D^2\phi>0$, the matrix $B$ is positive definite.  Moreover,
\[
 \det(B-\alpha I)=\kappa^2.
\]
Using $\kappa^2=1+\alpha^2$ and $\det D^2\phi=1$, we obtain
\[
 \det B=1+\alpha\tr B.
\]
If $\lambda_1,\lambda_2>0$ are the eigenvalues of $B$, then
\[
 \tan\bigl(\arctan\lambda_1+\arctan\lambda_2\bigr)
 =\frac{\lambda_1+\lambda_2}{1-\lambda_1\lambda_2}
 =-\frac1\alpha=\tan\Theta.
\]
Since $\det B>1$, the sum of the two eigenangles belongs to
$(\pi/2,\pi)$ and hence equals $\Theta$.
\end{proof}

We use a partial Legendre representation of the two-dimensional
Monge--Amp\`ere equation.  
Regard s as a parameter and let
$Q=Q(s,x,p)$ be harmonic in $(x,p)$ with $Q_{pp}<0$.  Define the partial Legendre transform 
\begin{equation}\label{eq:3d-partial-Legendre}
 y=-Q_p(s,x,p),\qquad
 \phi(s,x,y)=Q(s,x,p)+py,
\end{equation}
where $p=p(s,x,y)$ is determined by the first relation. 
\begin{lemma}[Partial Legendre identity]\label{lem:3d-partial-Legendre}
For the partial Legendre transform of the above $Q$, we have, for each fixed $s$,
\[
\det D_{x,y}^2 \phi = 1.
\]
Whenever $D_{x,y}^2 \phi$ is positive definite, $\phi$ is a strictly convex solution of the two-dimensional Monge--Ampère equation.
\end{lemma}

\begin{proof}
Differentiating \eqref{eq:3d-partial-Legendre} gives
\[
 \phi_y=p,\qquad \phi_x=Q_x,
\]
and hence
\[
 \phi_{yy}=-\frac1{Q_{pp}},\qquad
 \phi_{xy}=-\frac{Q_{xp}}{Q_{pp}},\qquad
 \phi_{xx}=Q_{xx}-\frac{Q_{xp}^2}{Q_{pp}}.
\]
It follows that
\[
 \det D^2_{x,y}\phi=-\frac{Q_{xx}}{Q_{pp}}=1,
\]
because $Q_{xx}+Q_{pp}=0$.  Since $Q_{pp}<0$, we have
$\phi_{yy}>0$; together with $\det D^2\phi=1$, this gives positive
definiteness.
\end{proof}

Take
\begin{equation}\label{eq:3d-Q-family}
 Q(s,x,p)=\frac{x^2-p^2}{2}
 +sH(x,p)+\frac{s^2}{2}J(x,p),
\end{equation}
where $H$ and $J$ are harmonic.  At $s=0$, differentiation of the
implicit relation in \eqref{eq:3d-partial-Legendre} gives
\begin{equation}\label{eq:3d-parameter-jets}
\begin{gathered}
 p=y,\qquad
 \phi=\frac{x^2+y^2}{2},\qquad
 \phi_s=H,\\
 D_{x,y}\phi_s=(H_x,H_p),\qquad
 \phi_{ss}=J+H_p^2.
\end{gathered}
\end{equation}
Indeed, $p_s=-\frac{Q_{sp}}{Q_{pp}}=H_p$ at $s=0$.
Here and below, expressions involving $H$, $J$, or their derivatives on
$s=0$ are evaluated at $(x,p)=(x,y)$.
Fix a closed disk in the $(x,y)$ variables and a slightly larger disk in
the $(x,p)$ variables.  Since $Q_{pp}=-1$ when $s=0$, after shrinking the
$s$-interval we have $Q_{pp}\le-1/2$ on the larger disk.  The map
$p\mapsto-Q_p(s,x,p)$ is then uniformly increasing, and
\eqref{eq:3d-partial-Legendre} defines $p=p(s,x,y)$ smoothly throughout
the smaller cylinder.

Define
\begin{equation}\label{eq:3d-U}
 U(s,x,y)
 =\frac{\alpha}{2}(x^2+y^2)
  +\kappa\phi(s,x,y)-\frac{\alpha+\kappa}{2}.
\end{equation}
After restricting to a sufficiently small neighborhood of $s=0$ and
the unit disk in the $(x,y)$ variables, the function $\phi(s,\cdot)$ is
strictly convex.  By Lemma \ref{lem:3d-phase-reduction},
\begin{equation}\label{eq:3d-exact-slice}
 \F_2(D^2_{x,y}U(s,\cdot))=\Theta
\end{equation}
for every such $s$.

\subsection{A local curvature model}

We first record the tangential Schur test used below.

\begin{lemma}[Tangential Schur test]\label{lem:3d-Schur}
Let $f=f(s,z)$, $z\in\R^2$, and write at a point
\[
 D^2f=
 \begin{pmatrix}
  a&b^{\mathsf T}\\
  b&B
 \end{pmatrix},
 \qquad
 Df=(g,q),
\]
where $B>0$ and $q\ne0$.  The restriction of $D^2f$ to the tangent
space of the level of $f$ is positive definite if and only if
\begin{equation}\label{eq:3d-Schur-quantity}
 \mathcal S_f:=
 a-b^{\mathsf T}B^{-1}b
 +\frac{(g-q^{\mathsf T}B^{-1}b)^2}
 {q^{\mathsf T}B^{-1}q}>0.
\end{equation}
\end{lemma}

\begin{proof}
A tangent vector with nonzero $s$ component may be normalized as
$(1,\xi)$, where $q\cdot\xi=-g$.  Completing the square gives
\[
 D^2f[(1,\xi),(1,\xi)]
 =a-b^{\mathsf T}B^{-1}b
 +(\xi+B^{-1}b)^{\mathsf T}B(\xi+B^{-1}b).
\]
Thus, we wish to minimize
\[
F(\xi):=(\xi+B^{-1}b)^T B(\xi+B^{-1}b)
\]
subject to the affine constraint
\[
q\cdot \xi = -g.
\]

Let
\[
\xi_0:=-B^{-1}b.
\]
Then
\[
F(\xi)=(\xi-\xi_0)^T B(\xi-\xi_0)=\|\xi-\xi_0\|_B^2,
\]
where $\|\cdot\|_B$ denotes the norm induced by the inner product
\[
\langle u,v\rangle_B:=u^T B v.
\]
Thus the problem is to find the point on the hyperplane
\[
H:=\{\xi: q\cdot \xi=-g\}
\]
that is closest to $\xi_0$ in the $B$-metric.

In the $B$-inner product, the normal direction to $H$ is given by $B^{-1}q$. Indeed, for any tangent vector $v$ to $H$ (i.e., $q\cdot v=0$), we have
\[
\langle B^{-1}q,v\rangle_B=(B^{-1}q)^T B v=q^T v=0.
\]
Therefore, the closest point $\xi^*\in H$ to $\xi_0$ must lie on the line
\[
\xi^*=\xi_0+tB^{-1}q
\]
for some $t\in\mathbb R$.

Since $\xi^*\in H$, we substitute into the constraint:
\[
q\cdot(\xi_0+tB^{-1}q)=-g.
\]
Using $\xi_0=-B^{-1}b$, we get
\[
-q^T B^{-1}b + t(q^T B^{-1}q)=-g.
\]
Solving for $t$,
\[
t=\frac{q^T B^{-1}b-g}{q^T B^{-1}q}.
\]

The minimal value of $F$ is therefore
\[
F_{\min}=\|\xi^*-\xi_0\|_B^2
      =\|tB^{-1}q\|_B^2
      =t^2(q^T B^{-1}q).
\]
Substituting the expression for $t$,
\[
F_{\min}
=\frac{(q^T B^{-1}b-g)^2}{(q^T B^{-1}q)^2}(q^T B^{-1}q)
=\frac{(g-q^T B^{-1}b)^2}{q^T B^{-1}q}.
\]

Hence the minimum of the quadratic form over the tangent space is
\[
a-b^T B^{-1}b+\frac{(g-q^T B^{-1}b)^2}{q^T B^{-1}q}.
\]
The remaining tangent direction lies
entirely in the $z$ variables and is positive because $B>0$.
\end{proof}

At $s=0$, write $z=(x,y)$ and $r=|z|$.  From
\eqref{eq:3d-parameter-jets}--\eqref{eq:3d-U},
\[
 B=(\alpha+\kappa)I,\qquad q=(\alpha+\kappa)z,\qquad
 g=\kappa H,\qquad b=\kappa DH,\qquad
 a=\kappa(J+H_p^2).
\]
Substitution into \eqref{eq:3d-Schur-quantity} yields
\begin{equation}\label{eq:3d-S-E}
 \mathcal S_U=\kappa J+(\kappa-\alpha)\kappa^2 E,
\end{equation}
where
\begin{equation}\label{eq:3d-E}
 E=-H_x^2+\gamma H_p^2
 +\frac{(H-z\cdot DH)^2}{r^2}.
\end{equation}

Choose an integer
\begin{equation}\label{eq:3d-m-choice}
 m\ge3,\qquad m\gamma>1,
\end{equation}
and put
\[
 H(x,p)=\operatorname{Re}(x+\ii p)^m.
\]
In polar coordinates $x+\ii p=re^{\ii\theta}$, a direct expansion gives
\begin{equation}\label{eq:3d-E-Fourier}
 E(r,\theta)=r^{2m-2}
 \left[
  A_0+A_-\cos\bigl(2(m-1)\theta\bigr)
  +A_+\cos(2m\theta)
 \right],
\end{equation}
where
\begin{equation}\label{eq:3d-A-coefficients}
 A_0=\frac{m^2\gamma-2m+1}{2},\qquad
 A_-=-\frac{m^2(1+\gamma)}2,\qquad
 A_+=\frac{(m-1)^2}{2}.
\end{equation}
Let
\begin{equation}\label{eq:3d-Ehat}
 \widehat E(x,p)
 =A_0+A_-\operatorname{Re}(x+\ii p)^{2m-2}
 +A_+\operatorname{Re}(x+\ii p)^{2m}
\end{equation}
be the harmonic extension of $E|_{\{r=1\}}$.  For $0<\rho<1$, set
\begin{equation}\label{eq:3d-Psi}
 \Psi(\rho)
 =A_0(1-\rho^{2m-2})
 +A_+\rho^{2m-2}(1-\rho^2).
\end{equation}
Then
\begin{equation}\label{eq:3d-Psi-limit}
 \lim_{\rho\uparrow1}\frac{\Psi(\rho)}{1-\rho}
 =m(m-1)(m\gamma-1)>0.
\end{equation}
Fix $\rho_*\in(0,1)$ sufficiently close to one that
\[
 \Psi_*:=\Psi(\rho_*)>0,
\]
and choose the second-order coefficient in \eqref{eq:3d-Q-family} as
\begin{equation}\label{eq:3d-J-choice}
 \eta=\frac{(\kappa-\alpha)\kappa}{2}\Psi_*,
 \qquad
 J=-(\kappa-\alpha)\kappa\,\widehat E+\eta.
\end{equation}

\begin{proposition}[The local model]\label{prop:3d-seed}
The central zero-level circle
\[
 \{s=0,\ U=0\}=\{s=0,\ |z|=1\}
\]
has positive definite tangential Hessian, and the corresponding Schur
quantity is
\begin{equation}\label{eq:3d-outer-margin}
 \mathcal S_U=\frac{(\kappa-\alpha)\kappa^2}{2}\Psi_*>0.
\end{equation}
On the other hand, at
\[
 p_*=(0,z_*),\qquad
 z_*=\rho_*
 \left(\cos\frac{\pi}{2m},\sin\frac{\pi}{2m}\right),
\]
one has
\begin{equation}\label{eq:3d-inner-level}
 U(p_*)=c_*:=\frac{\alpha+\kappa}{2}(\rho_*^2-1)<0,
\end{equation}
and the restriction of $D^2U$ to the tangent space of
$\{U=c_*\}$ has a negative direction.  More precisely, the Schur quantity
at $p_*$ is
\begin{equation}\label{eq:3d-inner-margin}
 -\frac{(\kappa-\alpha)\kappa^2}{2}\Psi_*<0.
\end{equation}
\end{proposition}

\begin{proof}
On $r=1$, \eqref{eq:3d-E-Fourier} and
\eqref{eq:3d-Ehat} give $E=\widehat E$.  Equations
\eqref{eq:3d-S-E} and \eqref{eq:3d-J-choice} therefore yield
\[
 \mathcal S_U=\kappa\eta
 =\frac{(\kappa-\alpha)\kappa^2}{2}\Psi_*.
\]
On the other tangential direction, which lies in the $z$-plane, the
quadratic form takes the value $\alpha+\kappa>0$.

Set $\theta_*=\pi/(2m)$.  Since $\cos(2m\theta_*)=-1$,
\eqref{eq:3d-E-Fourier}--\eqref{eq:3d-Psi} give
\[
 E(\rho_*,\theta_*)-\widehat E(\rho_*,\theta_*)=-\Psi_*.
\]
It follows from \eqref{eq:3d-S-E} and
\eqref{eq:3d-J-choice} that
\[
 \mathcal S_U(p_*)
 =\kappa\eta-(\kappa-\alpha)\kappa^2\Psi_*
 =-\frac{(\kappa-\alpha)\kappa^2}{2}\Psi_*.
\]
Finally, $D_zU(p_*)=(\alpha+\kappa)z_*\ne0$.  The negative level is regular at
$p_*$, and \cref{lem:3d-Schur} gives the required negative tangential
direction.
\end{proof}

\begin{remark}\label{rem:3d-explicit}
For $\Theta=3\pi/4$, one may take
\[
 \alpha=1,\quad \kappa=\sqrt2,\quad
 \gamma=\frac1{\sqrt2},\quad m=3,\quad \rho_*=\frac12.
\]
Then
\[
 \Psi_*=\frac{135\sqrt2-144}{64}>0,
\]
and the corresponding Schur values are
\[
 \pm\frac{414-279\sqrt2}{64}.
\]
\end{remark}

\subsection{Closing the zero level}

Choose $R>1$ and $\sigma>0$ so that $U$ is defined on a neighborhood of
the closed cylinder
\[
 \mathcal C=(-\sigma,\sigma)\times B_R,\qquad
 \overline{\mathcal C}=[-\sigma,\sigma]\times\overline{B_R}.
\]
Since $U(0,z)=\frac{\alpha+\kappa}{2}(|z|^2-1)$, we may shrink $\sigma$ and
choose $b_0,c_0>0$ such that
\begin{equation}\label{eq:3d-compact-bounds}
 D_z^2U\ge b_0I>0,\qquad
 U(s,0)\le-c_0,\qquad
 U(s,z)\ge c_0\quad\text{when }|z|=R.
\end{equation}
For $M>0$, define
\begin{equation}\label{eq:3d-quartic-cap}
 V_M(s,z)=U(s,z)+Ms^4.
\end{equation}
Let \(\Omega_M\) be the connected component containing \((0,0)\) of
\begin{equation}\label{eq:3d-Omega-M}
 \{V_M<0\}\cap\mathcal C.
\end{equation}

\begin{lemma}[Closing the model]\label{lem:3d-cap}
For every sufficiently large $M$, the component $\Omega_M$ defined in
\eqref{eq:3d-Omega-M} satisfies
\[
 \overline{\Omega_M}\Subset\mathcal C.
\]
Its boundary is smooth, connected and uniformly convex.

Moreover,
\begin{equation}\label{eq:3d-cap-phase}
 \F_2(D_z^2V_M)=\Theta
\end{equation}
throughout $\mathcal C$.  The point $p_*=(0,z_*)$ belongs to $\Omega_M$, and
\[
 V_M(p_*)=c_*,\qquad
 DV_M(p_*)=DU(p_*),\qquad
 D^2V_M(p_*)=D^2U(p_*).
\]
In particular, the Schur quantity of $V_M$ at $p_*$ is
\[
 \mathcal S_{V_M}(p_*)
 =-\frac{(\kappa-\alpha)\kappa^2}{2}\Psi_*<0,
\]
as in \eqref{eq:3d-inner-margin}.
\end{lemma}

\begin{proof}
Since $Ms^4$ is independent of $z$, one has
\[
 D_z^2V_M=D_z^2U.
\]
Thus \eqref{eq:3d-cap-phase} follows at once from
\eqref{eq:3d-exact-slice}.  The function $s^4$ and all its derivatives
of order at most two vanish at $s=0$.  Hence $V_M$, $DV_M$, and
$D^2V_M$ agree with $U$, $DU$, and $D^2U$, respectively, on the plane
$s=0$.  In particular,
\[
 V_M(p_*)=U(p_*)=c_*,
 \qquad
 \mathcal S_{V_M}(p_*)=\mathcal S_U(p_*).
\]
Also,
\[
 V_M(0,z)=U(0,z)=\frac{\alpha+\kappa}{2}(|z|^2-1).
\]
The section of $\{V_M<0\}$ by $s=0$ is therefore $B_1$.  Since
$|z_*|=\rho_*<1$, the point $p_*=(0,z_*)$ lies in the same component
of this section as $(0,0)$, and hence $p_*\in\Omega_M$.

We first describe the component $\Omega_M$.  For each fixed
$s\in[-\sigma,\sigma]$, the function $z\mapsto U(s,z)$ is uniformly
strictly convex.  By \eqref{eq:3d-compact-bounds}, its minimum over
$\overline{B_R}$ is attained in $B_R$.  Denote the unique minimizer by
$z(s)$.  It is characterized by
\[
 D_zU(s,z(s))=0.
\]
Since $D_z^2U\ge b_0I$, the implicit function theorem shows that
$z(s)$ is smooth.  The function
\[
 \mu(s):=U(s,z(s))=\min_{|z|\le R}U(s,z)
\]
is smooth as well.  We have $z(0)=0$ and
$\mu(0)=-(\alpha+\kappa)/2$.
Decreasing $\sigma$ once more, while retaining the notation
$\mathcal C$, we may assume that
\begin{equation}\label{eq:3d-mu-bounds}
 \frac{\alpha+\kappa}{4}\le -\mu(s)
 \le\frac{3(\alpha+\kappa)}{4}
 \qquad (|s|\le\sigma).
\end{equation}
The bounds in \eqref{eq:3d-compact-bounds} remain valid after this
restriction.  From now on, $\sigma$ is fixed before $M$ is chosen.

Set
\[
 C_U:=\max_{\overline{\mathcal C}}(-U)<\infty.
\]
If $(s,z)\in\mathcal C$ and $V_M(s,z)<0$, then
\[
 Ms^4<-U(s,z)\le C_U,
\]
and therefore
\begin{equation}\label{eq:3d-s-localization}
 |s|<\left(\frac{C_U}{M}\right)^{1/4}.
\end{equation}
In particular, if $M>C_U/\sigma^4$, then
\[
 V_M(\pm\sigma,z)\ge -C_U+M\sigma^4>0.
\]
On the other part of the artificial boundary,
\eqref{eq:3d-compact-bounds} gives
\[
 V_M(s,z)\ge U(s,z)\ge c_0
 \qquad (|z|=R).
\]
Thus the negative set stays away from $\partial\mathcal C$, and
\begin{equation}\label{eq:3d-relative-compactness}
 \overline{\Omega_M}\Subset\mathcal C.
\end{equation}

Define
\[
 h_M(s):=\mu(s)+Ms^4,
\]
and let
\[
 I_M=(s_-,s_+)
\]
be the component containing $0$ of the open set $\{h_M<0\}$.  For a
fixed $s$, put
\[
 D_s:=\{z\in B_R:U(s,z)<-Ms^4\}.
\]
The section $D_s$ is nonempty if and only if $h_M(s)<0$.  Whenever it
is nonempty, it is strictly convex and contains $z(s)$.  We claim that
\begin{equation}\label{eq:3d-section-description}
 \Omega_M
 =\bigcup_{s\in(s_-,s_+)}\bigl(\{s\}\times D_s\bigr).
\end{equation}
Indeed, the set on the right is connected.  A point $(s,z)$ in this
set can first be joined to $(s,z(s))$ by a segment contained in $D_s$;
it can then be joined to $(0,z(0))=(0,0)$ along the curve
$t\mapsto(t,z(t))$, because $h_M(t)<0$ for $t$ between $0$ and $s$.
Conversely, the projection onto the $s$-axis of any connected subset of
$\{V_M<0\}$ containing $(0,0)$ is a connected subset of $\{h_M<0\}$
containing $0$.  It must therefore lie in $I_M$.  This proves
\eqref{eq:3d-section-description}.

For large $M$, the preceding boundary estimates imply
$h_M(\pm\sigma)>0$.  Thus $s_-$ and $s_+$ belong to
$(-\sigma,\sigma)$ and satisfy
\[
 h_M(s_\pm)=0.
\]
Combining this identity with \eqref{eq:3d-mu-bounds}, we obtain
\begin{equation}\label{eq:3d-cap-scale}
 \left(\frac{\alpha+\kappa}{4M}\right)^{1/4}
 \le |s_\pm|
 \le
 \left(\frac{3(\alpha+\kappa)}{4M}\right)^{1/4}.
\end{equation}
In particular,
\[
 |s_\pm|\approx M^{-1/4},
\]
where \(\approx\) denotes two-sided comparison with constants independent
of \(M\).

At $s=s_\pm$, the open section $D_{s_\pm}$ is empty.  The corresponding
section of the closure is a single point:
\begin{equation}\label{eq:3d-cap-points}
 \overline{\Omega_M}\cap
 \bigl(\{s_\pm\}\times B_R\bigr)
 =\{P_\pm\},
 \qquad
 P_\pm:=(s_\pm,z(s_\pm)).
\end{equation}
To verify this assertion, the strong convexity in
\eqref{eq:3d-compact-bounds} gives
\begin{equation}\label{eq:3d-strong-convexity}
 U(s,z)-\mu(s)
 \ge\frac{b_0}{2}|z-z(s)|^2.
\end{equation}
On the other hand, if $(s,z)\in\overline{\Omega_M}$, then
\[
 U(s,z)-\mu(s)
 \le -Ms^4-\mu(s)=-h_M(s).
\]
Consequently,
\begin{equation}\label{eq:3d-section-collapse}
 |z-z(s)|^2
 \le\frac{2}{b_0}\bigl(-h_M(s)\bigr).
\end{equation}
The right-hand side tends to zero as $s\to s_\pm$ from inside $I_M$,
which proves \eqref{eq:3d-cap-points} and gives uniform collapse of the
sections.

We next prove that the two caps are smooth.  Differentiating
$\mu(s)=U(s,z(s))$ and using $D_zU(s,z(s))=0$, we find
\[
 \mu'(s)=U_s(s,z(s)).
\]
In particular, $\mu'$ is bounded independently of $M$.  By
\eqref{eq:3d-cap-scale},
\[
 4M|s_\pm|^3
 \ge 4\left(\frac{\alpha+\kappa}{4}\right)^{3/4}M^{1/4}.
\]
It follows, after increasing $M$ if necessary, that
\begin{equation}\label{eq:3d-cap-derivative}
 h_M'(s_-)<0<h_M'(s_+).
\end{equation}
At $P_\pm$ one has
\[
 D_zV_M(P_\pm)=0,
 \qquad
 (V_M)_s(P_\pm)
 =\mu'(s_\pm)+4Ms_\pm^3
 =h_M'(s_\pm)\ne0.
\]
Thus $0$ is a regular value of $V_M$ at the caps.  More precisely, the
implicit function theorem writes the zero level near
$P_\pm$ as a graph $s=\psi_\pm(z)$.  At $z=z(s_\pm)$,
\[
 D\psi_\pm=0,
 \qquad
 D^2\psi_\pm
 =-\frac{D_z^2U(P_\pm)}{h_M'(s_\pm)}.
\]
Thus the upper cap is a smooth strictly concave graph and the lower cap
is a smooth strictly convex graph.  The collapse in
\eqref{eq:3d-section-collapse} therefore does not produce a cusp.

At any other point of $\partial\Omega_M$, the $s$-coordinate belongs
to $(s_-,s_+)$.  If $D_zV_M=D_zU$ vanished at such a point, strict
convexity would give $z=z(s)$ and hence
\[
 V_M(s,z)=h_M(s)<0,
\]
which is impossible on the boundary.  Therefore $D_zV_M\ne0$ on the
side boundary.  We have proved that $\partial\Omega_M$ is a smooth
embedded closed surface.

We also record why this surface is connected.  For
$s\in(s_-,s_+)$ and $\omega\in\mathbb S^1$, consider
\[
 r\longmapsto U(s,z(s)+r\omega).
\]
Its derivative vanishes at $r=0$, while its second derivative is at
least $b_0$.  It is therefore strictly increasing for $r>0$ as long as
the ray remains in $B_R$.  Since $U(s,z(s))<-Ms^4$ and $U\ge c_0$ on
$\partial B_R$, the ray meets the level $U=-Ms^4$ exactly once.  Hence
\[
 \partial D_s
 =\{z(s)+r_M(s,\omega)\omega:\omega\in\mathbb S^1\},
\]
where $r_M(s,\omega)>0$ is smooth.  Estimate
\eqref{eq:3d-section-collapse} yields
\[
 \sup_{\omega\in\mathbb S^1}r_M(s,\omega)\longrightarrow0
 \qquad\text{as }s\to s_\pm.
\]
Thus the family of side curves closes at $P_-$ and $P_+$, and
$\partial\Omega_M$ is connected.

It remains to prove strict convexity.  On the side boundary set
\[
 B=D_z^2U,\qquad q=D_zU,
 \qquad d=q^{\mathsf T}B^{-1}q,
\]
and
\[
 a_0=U_{ss}-(D_zU_s)^{\mathsf T}B^{-1}D_zU_s,
 \qquad
 a_1=U_s-q^{\mathsf T}B^{-1}D_zU_s.
\]
Here $B\ge b_0I$ and $q\ne0$, so $d>0$.  Applying
\cref{lem:3d-Schur} to $V_M$ gives
\begin{equation}\label{eq:3d-cap-Schur}
 \mathcal S_M
 =a_0+12Ms^2+\frac{(a_1+4Ms^3)^2}{d}.
\end{equation}
All quantities involving only $U$ are controlled on the fixed cylinder.
In particular, there are constants $C_1,A_1>0$, independent of $M$,
such that
\begin{equation}\label{eq:3d-a-bounds}
 a_0\ge-C_1,
 \qquad
 |a_1|\le A_1
 \quad\text{on }\overline{\mathcal C}.
\end{equation}
Choose $L>0$, independently of $M$, so that
\[
 12L^2>C_1+1.
\]
If a side boundary point satisfies $|s|\ge LM^{-1/2}$, then the last
term in \eqref{eq:3d-cap-Schur} is nonnegative and
\[
 \mathcal S_M
 \ge-C_1+12Ms^2
 \ge-C_1+12L^2>1.
\]

It remains to consider side boundary points for which
\[
 |s|\le LM^{-1/2}.
\]
At such a point, $V_M(s,z)=0$, and hence
\[
 U(s,z)=-Ms^4.
\]
Let $K_0=\|U_s\|_{L^\infty(\overline{\mathcal C})}$.  Using the central
identity for $U$, we obtain
\begin{align*}
 \frac{\alpha+\kappa}{2}\bigl||z|^2-1\bigr|
 &=|U(0,z)|\\
 &\le |U(0,z)-U(s,z)|+|U(s,z)|\\
 &\le K_0|s|+Ms^4\\
 &\le K_0LM^{-1/2}+L^4M^{-1}.
\end{align*}
Thus these boundary points converge uniformly, as $M\to\infty$, to the
central circle $\{s=0,\ |z|=1\}$.

Put
\[
 \mathfrak s_*:=\frac{(\kappa-\alpha)\kappa^2}{2}\Psi_*>0.
\]
On the central circle, $d=\alpha+\kappa$ and
$\mathcal S_U=\mathfrak s_*$.  By continuity, for all sufficiently
large $M$, the points under consideration satisfy
\begin{equation}\label{eq:3d-central-margin}
 d\ge d_0>0,
 \qquad
 \mathcal S_U:=a_0+\frac{a_1^2}{d}
 \ge\frac34\mathfrak s_*.
\end{equation}
Moreover,
\begin{align*}
 \mathcal S_M-\mathcal S_U
 &=12Ms^2+
   \frac{8Ms^3a_1+16M^2s^6}{d}\\
 &\ge-\frac{8A_1}{d_0}M|s|^3\\
 &\ge-\frac{8A_1L^3}{d_0}M^{-1/2}.
\end{align*}
After increasing $M$ once more, the last quantity is bounded below by
$-\mathfrak s_*/4$.  Therefore
\[
 \mathcal S_M\ge\frac12\mathfrak s_*>0
\]
throughout the central layer.  Together with the preceding estimate,
this proves positivity of the tangential Hessian at every point of the
side boundary.

At a cap point $P_\pm$, the gradient of $V_M$ is parallel to the
$s$-axis.  If $v=(\tau,\xi)$ is tangent to $\partial\Omega_M$ at
$P_\pm$, then \eqref{eq:3d-cap-derivative} gives $\tau=0$.  Therefore
\[
 D^2V_M(P_\pm)[v,v]
 =\xi^{\mathsf T}D_z^2U(P_\pm)\xi
 \ge b_0|\xi|^2>0
\]
for every nonzero tangent vector.  Thus the tangential Hessian is
positive definite at the caps as well.

Since $\Omega_M=\{V_M<0\}$ locally along its boundary, the outward unit
normal is
\[
 \nu=\frac{DV_M}{|DV_M|}.
\]
For a tangent vector $v$,
\[
 \mathrm{II}_{\nu}(v,v)
 =\frac{D^2V_M[v,v]}{|DV_M|}.
\]
The second fundamental form of $\partial\Omega_M$ is therefore positive
definite everywhere.  The Hadamard theorem for compact locally strictly
convex hypersurfaces now shows that $\partial\Omega_M$ bounds a strictly
convex body.  This body is $\overline{\Omega_M}$, because
$\partial\Omega_M$ is the boundary of the bounded component
$\Omega_M$.

Finally, after $M$ has been fixed, compactness of $\partial\Omega_M$
and continuity of the second fundamental form give
\[
 \kappa_M
 :=\min_{p\in\partial\Omega_M}
   \min_{\substack{v\in T_p\partial\Omega_M\\ |v|=1}}
   \mathrm{II}_{\nu}(v,v)>0.
\]
This is the asserted uniform convexity.
\end{proof}

The geometry of the preceding lemma is summarized in
\cref{fig:omega-M}.

\begin{figure}[!htbp]
\centering
\resizebox{0.92\textwidth}{!}{%
\begin{tikzpicture}[x=1cm,y=1cm]

\begin{scope}[shift={(4.05,3.50)}]
  \node[panel title] at (0,2.72)
    {(a) The component $\Omega_M\Subset\mathcal C$};

  \draw[artificial] (-3.55,2.05) -- (3.55,2.05);
  \draw[artificial] (-3.55,-2.05) -- (3.55,-2.05);
  \draw[artificial] (-3.55,0) ellipse [x radius=0.36,y radius=2.05];
  \draw[artificial] ( 3.55,0) ellipse [x radius=0.36,y radius=2.05];
  \node[text=CylinderGray,anchor=south west] at (-3.45,2.10)
    {$\mathcal C=(-\sigma,\sigma)\times B_R$};

  \path[fill=OmegaFill]
    (-2.72,0.10)
      .. controls (-2.55,0.92) and (-1.30,1.48) .. (0,1.58)
      .. controls (1.25,1.48) and (2.52,0.93) .. (2.92,-0.10)
      .. controls (2.50,-0.88) and (1.20,-1.38) .. (0,-1.46)
      .. controls (-1.28,-1.38) and (-2.53,-0.86) .. (-2.72,0.10)
    -- cycle;
  \draw[boundary]
    (-2.72,0.10)
      .. controls (-2.55,0.92) and (-1.30,1.48) .. (0,1.58)
      .. controls (1.25,1.48) and (2.52,0.93) .. (2.92,-0.10)
      .. controls (2.50,-0.88) and (1.20,-1.38) .. (0,-1.46)
      .. controls (-1.28,-1.38) and (-2.53,-0.86) .. (-2.72,0.10);

  \draw[hidden section] (-1.32,0.12)
    ellipse [x radius=0.27,y radius=1.27];
  \path[fill=OmegaBlue!9,draw=none] (0,0)
    ellipse [x radius=0.34,y radius=1.52];
  \draw[section] (0,0)
    ellipse [x radius=0.34,y radius=1.52];
  \draw[hidden section] (1.48,-0.06)
    ellipse [x radius=0.25,y radius=1.15];

  \draw[CylinderGray,densely dotted,line width=0.65pt]
    (-2.72,0.10) .. controls (-1.60,0.17) and (-0.58,0.08) .. (0,0)
    .. controls (0.82,-0.04) and (1.85,-0.09) .. (2.92,-0.10);

  \draw[axis] (-3.95,0) -- (4.03,0) node[right] {$s$};
  \draw[CylinderGray,densely dotted] (-2.72,0) -- (-2.72,0.10);
  \draw[CylinderGray,densely dotted] ( 2.92,0) -- ( 2.92,-0.10);
  \fill[OmegaBlue!88!black] (-2.72,0.10) circle (1.2pt);
  \fill[OmegaBlue!88!black] ( 2.92,-0.10) circle (1.2pt);
  \node[above left=-1pt] at (-2.72,0.10) {$P_-$};
  \node[above right=-1pt] at (2.92,-0.10) {$P_+$};
  \node[below=6pt] at (-2.72,0) {$s_-$};
  \node[below=6pt] at ( 2.92,0) {$s_+$};
  \node[below=2pt] at (0,0) {$0$};

  \node[text=OmegaBlue!75!black] at (0.72,0.58) {$\Omega_M$};
  \node[anchor=south west] (bdry) at (1.76,1.72) {$V_M=0$};
  \draw[note arrow] (bdry.south) -- (2.12,0.83);

  \node[anchor=south west,fill=white,fill opacity=0.78,text opacity=1,
        inner sep=1pt] (d0) at (0.38,1.30) {$D_0=B_1$};
  \draw[note arrow] (d0.west) -- (0.08,1.08);
  \node[anchor=north] (ds) at (-1.75,-1.68) {$D_s$};
  \draw[note arrow] (ds.north east) -- (-1.33,-0.88);

\end{scope}

\begin{scope}[shift={(13.25,3.50)}]
  \node[panel title] at (0,2.72)
    {(b) A representative meridian view};

  \draw[artificial] (-3.45,-2.05) rectangle (3.45,2.05);
  \node[text=CylinderGray,anchor=south east] at (3.38,2.08)
    {$|s|<\sigma$, $|z|<R$};

  \path[fill=OmegaFill,boundary]
    (-2.72,0)
      .. controls (-2.54,0.88) and (-1.35,1.43) .. (0,1.49)
      .. controls (1.23,1.44) and (2.54,0.88) .. (2.92,0)
      .. controls (2.52,-0.84) and (1.25,-1.34) .. (0,-1.39)
      .. controls (-1.30,-1.34) and (-2.55,-0.82) .. (-2.72,0)
    -- cycle;

  \draw[axis] (-3.82,0) -- (3.86,0) node[right] {$s$};
  \draw[axis] (0,-1.82) -- (0,2.27) node[above] {$\zeta$};

  \draw[section] (0,-1.39) -- (0,1.49);
  \draw[hidden section] (-1.34,-1.15) -- (-1.34,1.20);
  \node[anchor=west] at (-1.28,1.02) {$D_s\cap\Pi$};
  \node[anchor=north east] at (-0.10,-1.45) {$D_0\cap\Pi$};

  \fill[OmegaBlue!88!black] (-2.72,0) circle (1.2pt);
  \fill[OmegaBlue!88!black] ( 2.92,0) circle (1.2pt);
  \draw[CylinderGray,line width=0.55pt] (-2.72,-0.38) -- (-2.72,0.38);
  \draw[CylinderGray,line width=0.55pt] ( 2.92,-0.38) -- ( 2.92,0.38);
  \node[below=3pt,fill=white,inner sep=0.5pt] at (-2.72,0) {$s_-$};
  \node[below=3pt,fill=white,inner sep=0.5pt] at ( 2.92,0) {$s_+$};
  \node[below left=2pt] at (0,0) {$0$};

  \node[anchor=north west,align=left,font=\scriptsize,
        fill=white,fill opacity=0.76,text opacity=1,inner sep=1pt]
    (regular) at (0.42,1.36)
    {$D_zV_M=0$,\quad $(V_M)_s\ne0$};
  \draw[note arrow] (regular.east) -- (2.84,0.16);

  \node[anchor=east,align=right,font=\scriptsize,
        fill=white,fill opacity=0.76,text opacity=1,inner sep=1pt]
    (cap) at (2.25,-0.86)
    {$s=s_+-Q_+|\zeta-\zeta_+|^2$\\
     ${}+O(|\zeta-\zeta_+|^3)$};
  \draw[note arrow] (cap.east) -- (2.78,-0.22);

  \node[anchor=north,align=center,font=\scriptsize] at (0,-2.23)
    {$D_s\ne\varnothing\iff h_M(s)=\mu(s)+Ms^4<0$,\\[-1pt]
     $|s_\pm|\approx M^{-1/4}$};
\end{scope}

\end{tikzpicture}%
}
\caption{Schematic geometry of $\Omega_M$.  Every nonempty $z$-section
is a strictly convex disk and collapses smoothly at $s=s_\pm$.  No
rotational symmetry is assumed.}
\label{fig:omega-M}
\end{figure}
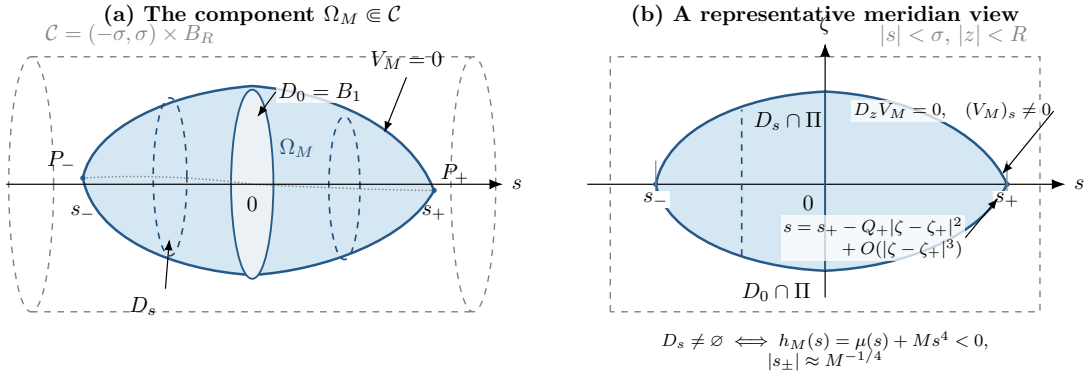
\FloatBarrier

\subsection{Anisotropic rescaling and comparison}

Fix $M$ as in \cref{lem:3d-cap}.  For $0<\e\ll1$, set
\begin{equation}\label{eq:3d-stretch}
 W_\e(X,z)=V_M(\e X,z),\qquad
 \Omega_\e=\{(X,z):(\e X,z)\in\Omega_M\}.
\end{equation}
The domain $\Omega_\e$ is smooth, bounded, and uniformly convex, and
\[
 W_\e=0\quad\text{on }\partial\Omega_\e,\qquad
 W_\e<0\quad\text{in }\Omega_\e.
\]

\begin{lemma}[Phase residual]\label{lem:3d-residual}
Uniformly on $\overline{\Omega_\e}$,
\[
 \F_3(D^2W_\e)=\Theta+O(\e^2).
\]
\end{lemma}

\begin{proof}
At a fixed point $(s,z)\in\overline{\Omega_M}$ in the unscaled
variables, write
\[
\begin{gathered}
 G_\e(s,z)=
 \begin{pmatrix}
  \e^2a&\e b^{\mathsf T}\\
  \e b&B
 \end{pmatrix},\\
 a=(V_M)_{ss},\qquad b=D_z(V_M)_s,\qquad B=D_z^2V_M>0.
\end{gathered}
\]
For $s=\e X$, one has $D^2W_\e(X,z)=G_\e(s,z)$.
At $\e=0$, this matrix is $\operatorname{diag}(0,B)$ and has phase
$\Theta$.  If $S=\operatorname{diag}(-1,I_2)$, then
$G_{-\e}=S G_\e S$.  The phase is invariant under orthogonal
conjugation, so $\F_3(G_\e(s,z))$ is an even smooth function of $\e$.
The operator $\F_3$ is smooth on $\operatorname{Sym}(3)$, and all the
coefficients above vary on the fixed compact set
$\overline{\Omega_M}$.  Taylor's formula therefore gives the uniform
$O(\e^2)$ estimate.
\end{proof}

\begin{proposition}[Comparison with the Dirichlet solution]
\label{prop:3d-barriers}
Let $u_\e$ be the solution of
\begin{equation}\label{eq:3d-Dirichlet-epsilon}
 \begin{cases}
  \F_3(D^2u_\e)=\Theta&\text{in }\Omega_\e,\\
  u_\e=0&\text{on }\partial\Omega_\e.
 \end{cases}
\end{equation}
There is a constant $C$, independent of $\e$, such that
\begin{equation}\label{eq:3d-C0-close}
 \|u_\e-W_\e\|_{L^\infty(\Omega_\e)}\le C\e^2.
\end{equation}
\end{proposition}

\begin{proof}
Existence and uniqueness follow from \cite[Theorem~1.2]{luDCDS}.
The solution belongs to
$C^{0,1}(\overline{\Omega_\e})$ and is smooth in $\Omega_\e$.

Choose $R_0>R$ and set
\[
 \psi(X,z)=R_0^2-|z|^2>0
 \quad\text{on }\overline{\Omega_\e}.
\]
Put $P=\operatorname{diag}(0,I_2)$.  The linearization of the phase
operator is
\[
 D\F_3\big|_G[N]
 =\tr\bigl((I+G^2)^{-1}N\bigr).
\]
Since $B=D_z^2V_M$ ranges over a compact subset of the positive definite
cone,
\[
 \lambda_0:=\min_{\overline{\Omega_M}}
 \tr\bigl((I+B^2)^{-1}\bigr)>0.
\]
Let the residual in \cref{lem:3d-residual} be bounded by
$C_{\mathrm{res}}\e^2$, and fix
$A>2C_{\mathrm{res}}/\lambda_0$.  By continuity, after decreasing
$\e$ we have
\[
 D\F_3\big|_{D^2W_\e+tP}[P]\ge\frac{\lambda_0}{2}
 \qquad (|t|\le2A\e^2)
\]
uniformly on $\overline{\Omega_\e}$.  Define
\[
 \underline W_\e=W_\e-A\e^2\psi,\qquad
 \overline W_\e=W_\e+A\e^2\psi.
\]
Since $D^2\psi=-2P$, the integral form of the linearization gives
\[
 \F_3(D^2\underline W_\e)
 \ge\Theta+(A\lambda_0-C_{\mathrm{res}})\e^2>\Theta,
\]
\[
 \F_3(D^2\overline W_\e)
 \le\Theta-(A\lambda_0-C_{\mathrm{res}})\e^2<\Theta.
\]
On $\partial\Omega_\e$,
\[
 \underline W_\e\le0=u_\e\le\overline W_\e.
\]
Indeed, if $\underline W_\e-u_\e$ had a positive maximum at an interior
point, then
$D^2\underline W_\e\le D^2u_\e$ there.  The ellipticity of $\F_3$
would give
$\F_3(D^2\underline W_\e)\le\F_3(D^2u_\e)=\Theta$, a contradiction.
The same argument applied to $u_\e-\overline W_\e$ gives
\[
 \underline W_\e\le u_\e\le\overline W_\e
\quad\text{in }\Omega_\e.
\]
Since $0<\psi\le R_0^2$, estimate \eqref{eq:3d-C0-close} follows.
\end{proof}

\subsection{Proof of the three-dimensional theorem}

\begin{lemma}[A midpoint criterion]\label{lem:3d-midpoint}
Let $f\in C^3$ near $p$, assume $Df(p)\ne0$, and suppose that
\[
 Df(p)\cdot\tau=0,\qquad
 D^2f(p)[\tau,\tau]=-\lambda<0
\]
for a unit vector $\tau$.  Then there are $h>0$, $\ell<f(p)$, and
$\Delta>0$ such that
\[
 f(p\pm h\tau)\le\ell-4\Delta,\qquad
 f(p)\ge\ell+4\Delta.
\]
\end{lemma}

\begin{proof}
Taylor's formula gives
\[
 f(p\pm h\tau)
 =f(p)-\frac{\lambda}{2}h^2+O(h^3).
\]
For small $h$, the right-hand side is at most
$f(p)-\lambda h^2/4$.  Taking
\[
 \ell=f(p)-\frac{\lambda h^2}{8},\qquad
 \Delta=\frac{\lambda h^2}{32}
\]
proves the assertion.
\end{proof}

\begin{proof}[Proof of \cref{thm:three-dimensional}]
By \cref{prop:3d-seed,lem:3d-cap}, the level of $V_M$ through $p_*$
has a negative tangential direction, while the component $\Omega_M$ in
\eqref{eq:3d-Omega-M} is smooth and uniformly convex.  Apply
\cref{lem:3d-midpoint} to obtain
\[
 p_\pm=p_*\pm h\tau,\qquad \ell<0,\qquad \Delta>0,
\]
with all three points in $\Omega_M$, such that
\begin{equation}\label{eq:3d-model-gap}
 V_M(p_\pm)\le\ell-4\Delta,\qquad
 V_M(p_*)\ge\ell+4\Delta.
\end{equation}

The linear map
\[
 T_\e(s,z)=\left(\frac{s}{\e},z\right)
\]
preserves midpoints.  Put
\[
 q_\pm=T_\e(p_\pm),\qquad q_*=T_\e(p_*).
\]
Then $q_*=(q_++q_-)/2$ and
$W_\e(q)=V_M(T_\e^{-1}q)$.  By
\cref{prop:3d-barriers}, for small $\e$ the difference
$|u_\e-W_\e|$ at these three points is less than $\Delta$.  Hence
\[
 u_\e(q_\pm)<\ell,\qquad u_\e(q_*)>\ell.
\]
Thus $\{u_\e<\ell\}$ is not convex.

The strict inequalities remain valid throughout the interval
\[
 I_\e=
 \bigl(\max\{u_\e(q_+),u_\e(q_-)\},u_\e(q_*)\bigr).
\]
The maximum principle gives $u_\e<0$ in $\Omega_\e$, so
$I_\e\subset(-\infty,0)$.  Sard's theorem provides a regular value in
$I_\e$.
Renaming
$\Omega_\e$ and $u_\e$ as $\Omega_\Theta$ and $u_\Theta$ completes the
proof.
\end{proof}

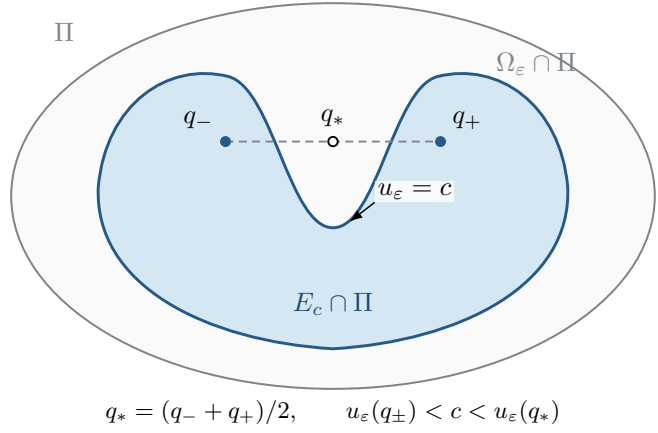
\begin{figure}[t]
\centering
\begin{tikzpicture}[
  x=1cm,y=1cm,
  domain boundary/.style={
    draw=CylinderGray,line width=0.75pt
  },
  level boundary/.style={
    draw=OmegaBlue!88!black,line width=1.1pt
  },
  midpoint line/.style={
    draw=CylinderGray,densely dashed,line width=0.7pt
  },
  point in/.style={
    circle,fill=OmegaBlue!88!black,inner sep=1.45pt
  },
  point out/.style={
    circle,draw=black,fill=white,
    line width=0.7pt,inner sep=1.25pt
  },
  label arrow/.style={
    -{Latex[length=1.7mm]},line width=0.55pt
  },
  every node/.style={font=\small}
]

  \path[fill=black!2,domain boundary]
    (0,0) ellipse [x radius=4.25,y radius=2.55];

  \node[anchor=north east,text=CylinderGray]
    at (3.35,2.03)
    {$\Omega_\varepsilon\cap\Pi$};

  \node[anchor=south west,text=CylinderGray]
    at (-3.82,1.92) {$\Pi$};

  \path[fill=OmegaFill,level boundary]
    (-3.10,0.12)
      .. controls (-3.00,1.22) and (-2.18,1.78)
      .. (-1.38,1.58)
      .. controls (-0.79,1.42) and (-0.58,-0.42)
      .. (0,-0.42)
      .. controls (0.58,-0.42) and (0.79,1.42)
      .. (1.38,1.58)
      .. controls (2.18,1.78) and (3.00,1.22)
      .. (3.10,0.12)
      .. controls (3.16,-1.18) and (1.68,-1.92)
      .. (0,-2.02)
      .. controls (-1.68,-1.92) and (-3.16,-1.18)
      .. (-3.10,0.12)
    -- cycle;

  \node[text=OmegaBlue!72!black]
    at (0,-1.43) {$E_c\cap\Pi$};

  \coordinate (qm) at (-1.42,0.72);
  \coordinate (qs) at (0,0.72);
  \coordinate (qp) at (1.42,0.72);

  \draw[midpoint line] (qm) -- (qp);

  \node[point in]  at (qm) {};
  \node[point out] at (qs) {};
  \node[point in]  at (qp) {};

  \node[above left=1pt]  at (qm) {$q_-$};
  \node[above=2pt]       at (qs) {$q_*$};
  \node[above right=1pt] at (qp) {$q_+$};

  \node[
    anchor=west,
    fill=white,
    fill opacity=0.84,
    text opacity=1,
    inner sep=1pt
  ] (level) at (0.55,0.06)
    {$u_\varepsilon=c$};

  \draw[label arrow]
    (level.south west) -- (0.22,-0.34);

  \node[align=center,font=\footnotesize]
    at (0,-2.86)
    {$q_*=(q_-+q_+)/2$,\qquad
      $u_\varepsilon(q_\pm)<c<u_\varepsilon(q_*)$};

\end{tikzpicture}

\caption{A planar section exhibiting the midpoint defect.
The endpoints $q_-$ and $q_+$ belong to the sublevel set
$E_c=\{u_\varepsilon<c\}$, whereas their midpoint $q_*$ does not.
Thus $E_c$ is not convex. Here $c\in I_\varepsilon$ is chosen to be
a regular value.}
\label{fig:3d-midpoint-defect}
\end{figure}


\end{document}